\documentclass[11pt]{amsart}
   \usepackage[margin=3.55cm]{geometry}
\numberwithin{equation}{section}
\usepackage{amsmath, amsfonts,amsthm,amssymb,amscd, verbatim,graphicx,color,multirow,booktabs, caption,tikz,tikz-cd, mathdots,bm}
\usepackage{tikz-cd}
 \usepackage[pagebackref]{hyperref} 
\usetikzlibrary{positioning} 

\newtheorem{theorem}{Theorem}[section]
\newtheorem{lemma}[theorem]{Lemma}

 \newtheorem{corollary}[theorem]{Corollary}

      \theoremstyle{definition}
     \newtheorem*{definition}{Definition}
     \newtheorem{example}[theorem]{Example}
     
     \theoremstyle{remark}
     \newtheorem{remark}[theorem]{Remark}

\newcommand{\Sym}{\mathop{\mathrm{Sym}}}
\newcommand{\Alt}{\mathop{\mathrm{Alt}}}
\newcommand{\Aut}{\mathop{\mathrm{Aut}}}
\newcommand{\Out}{\mathop{\mathrm{Out}}}

\newcommand{\AGL}{\mathop{\mathrm{AGL}}}

\newcommand{\PSL}{\mathop{\mathrm{PSL}}}
\newcommand{\PSU}{\mathop{\mathrm{PSU}}}
\newcommand{\Stab}{\mathop{\mathrm{Stab}}}
\newcommand{\StabG}{\mathop{\mathrm{Stab}_G}}
\newcommand{\StabA}{\mathop{\mathrm{Stab}_A}}
\newcommand{\StabB}{\mathop{\mathrm{Stab}_B}}
\newcommand{\StabS}{\mathop{\mathrm{Stab}_S}}
\newcommand{\Ker}{\mathop{\mathrm{Ker}}}
\newcommand{\Fi}{\mathop{\mathrm{Fi}}}

 \definecolor{mycolor}{rgb}{0.55,0.0,0.16}
  \definecolor{myred}{rgb}{0.75,0.0,0.16} 
  \definecolor{mygreen}{rgb}{0.0,0.4,0.16} 
  \definecolor{myviolet}{rgb}{1,0,1} 
   \definecolor{mypink}{rgb}{0.9,0,0.5}
	
	\newcommand{\MYhref}[3][black]{\href{#2}{\color{#1}{#3}}}%

\hypersetup{
colorlinks=true,
linkcolor=true,
linktocpage=true,
pageanchor=true,
hyperindex=true
}

 \AtBeginDocument{
     \hypersetup{
  linkcolor=mycolor,
  urlcolor=mypink,
citecolor=mygreen
}
     }

\makeatletter
\@namedef{subjclassname@2020}{%
  \textup{2020} Mathematics Subject Classification}
\makeatother


\author[Luca Sabatini]{Luca Sabatini}
\address{\parbox{\linewidth}{Luca Sabatini, Mathematics Institute, Zeeman Building, University of Warwick\\
Coventry CV4\,7AL, United Kingdom \vspace{0.1cm}}}
\email{luca.sabatini@warwick.ac.uk, sabatini.math@gmail.com} 

\begin{document} 
 \title[Stabilizers in finite permutation groups]{On stabilizers in finite permutation groups} 

\maketitle 

\begin{abstract} 
Let $G$ be a permutation group on the finite set $\Omega$.
We prove various results about partitions of $\Omega$ whose stabilizers have good properties.
In particular,
in every solvable permutation group there is a set-stabilizer whose orbits have length at most $6$,
which is best possible and answers two questions of Babai.
Every solvable maximal subgroup of any almost simple group has derived length at most $10$, which is best possible.
In every primitive group with solvable stabilizer,
there are two points whose stabilizer has derived length bounded by an absolute constant.
\end{abstract} 


\vspace{0.5cm}
\section{Introduction} 

Let $G$ be a permutation group on the finite set $\Omega$.
One of the most studied topics in permutation group theory is finding a partition of $\Omega$ whose stabilizer is trivial.
For example, a subset $\Delta \subseteq \Omega$ is said to be in a regular orbit if its setwise stabilizer is trivial.
In general, the existence of such a $\Delta$ is far from certain \cite{Glu83,Ser97}, and it is necessary to stabilize many subsets.
On the other hand, the base size is the minimum number of points such that their pointwise stabilizer is trivial.
This corresponds to fixing a partition made by some single points and the complement.
Finding upper bounds for the base size is a difficult problem that has produced an immense volume of research,
and some of the most impressive results obtained are \cite{Bab81,Pyb93,Ser96,DHM18}.
For more details on these topics, we refer the reader to the pleasant survey \cite{BC11}.

In this paper, we are interested in a small variation on these questions.
More precisely, we fix only the bare minimum (so one subset, or a pair of points)
and we do not aim to find a trivial stabilizer, as that would be impossible in general,
but to find a stabilizer that has at least some good properties.
Regarding subsets, Babai \cite{Bab22} has showed that in every solvable permutation group
there exists a set-stabilizer whose orbits have length bounded by a nonexplicit absolute constant.
Answering \cite[Quest. 7b]{Bab22}, our first main result provides the best possible constant:

\begin{theorem} \label{thSolvSet} 
Let $G \leqslant \Sym(\Omega)$ be a solvable permutation group.
Then there exists a subset of $\Omega$ whose setwise stabilizer has all orbits of length at most $6$.
The constant $6$ is best possible.
\end{theorem} 

The following consequence is immediate.

\begin{corollary} \label{corMain}
Let $G \leqslant \Sym(\Omega)$ be a solvable permutation group.
Then there exists a subset of $\Omega$ whose setwise stabilizer has derived length at most $3$.
\end{corollary}

It happens that metabelian stabilizers need not exist in general
(see Example \ref{exBad}), so Corollary \ref{corMain} answers \cite[Quest. 7a]{Bab22}.
Incidentally, the same example answers negatively a question in \cite{Glu25} on the existence of supersolvable set-stabilizers.
We stress that the two constants $6$ and $3$ above turned out to be higher than initially expected,
and better results can be obtained for primitive solvable groups.
As far as we know, Theorem \ref{thSolvSet} is the first statement concerning stabilizers of imprimitive groups
where the property involved is not preserved by a group extension,
and part of the difficulty of the proof is in finding the right statement to prove by induction.
This strategy, combined with a detailed analysis of primitive groups, turned out to be both powerful and flexible.

With similar ideas, we are able to show that, in every solvable permutation group,
stabilizing two disjoint subsets can force the orbits to have length at most $2$.
This gives a strong positive answer to \cite[Quest. 7c]{Bab22}.
Moreover, we give elementary proofs concerning nonsolvable groups:
we provide an easier proof of a theorem of Cameron \cite{Cam86}, and answer \cite[Quest. 6a]{Bab22}.
In particular, we prove the following:

\begin{theorem} \label{thNonAlt}
Let $G$ be any permutation group on a finite set.
Then there exists a set-stabilizer whose nonalternating composition factors have order bounded by an absolute constant.
\end{theorem}

This is new even using the Classification of Finite Simple Groups.

We come to the question of fixing two points.
Extending the work of Seress \cite{Ser96},
Burness \cite{Bur21} showed that in a primitive group with solvable stabilizer
there always exist $5$ points whose pointwise stabilizer is trivial.
The hypothesis of having a solvable stabilizer is natural,
 if we aim to find structural results on the stabilizer of two points.
For primitive solvable groups, in \cite{LS25} the authors observed
that there exist two points whose pointwise stabilizer has bounded derived length.
On the other hand, motivated by various applications,
Li and Zhang \cite{LZ11} have classified the solvable maximal subgroups of the almost simple groups.
Combining their work with some properties of the finite simple groups, we remark the following:

\begin{theorem} \label{thMaxAS}
 Let $G$ be an almost simple group, and let $H<G$ be a maximal subgroup.
 If $H$ is solvable, then the derived length of $H$ is at most $10$.
 The constant $10$ is attained only by the solvable maximal subgroups of the Fischer group $\Fi_{_{23}}$.
 \end{theorem} 

Despite Theorem \ref{thMaxAS},
it is easy to exhibit primitive nonsolvable groups of product type with solvable stabilizer of arbitrarily large derived length.
Dealing with the product type case, we are able to answer in the affirmative the question at the end of \cite{LS25}:

\begin{theorem} \label{thMain} 
Let $G$ be a primitive group with solvable stabilizer.
Then there exist two points whose pointwise stabilizer has derived length bounded by an absolute constant.
\end{theorem} 

The proof of Theorem \ref{thMain} is an application of Corollary \ref{corMain}
(or of the nonexplicit result of Babai),
and it is perhaps surprising that, to answer a question about pointwise stabilizers of primitive nonsolvable groups,
one has to deal with set-stabilizers of arbitrary solvable permutation groups.
The idea, usually applied in the context of imprimitive linear groups,
is using a partition to construct a vector with small centralizer (see the book \cite{MW93} for various examples).
Our argument involves a double coset decomposition, and a crude bound gives the constant $13$.
It is likely that a more in-depth study of almost simple groups leads to an improvement,
but the best possible constant is not even known for solvable groups \cite[Rem. 2.9]{LS25}.

The paper is simply organized.
In Section \ref{sec2} we study set-stabilizers, proving Theorems \ref{thSolvSet} and \ref{thNonAlt} and other related results.
In Section \ref{sec3}, we prove Theorems \ref{thMaxAS} and \ref{thMain}.
The methods are a mixture of group theory and combinatorics.

\vspace{0.3cm} \noindent
{\bfseries Acknowledgments:}
 I am grateful to S. Eberhard for many useful conversations.
 After the public announcement of a first version of this paper, I received precious comments from D. Gluck, H.Y. Huang and S.V. Skresanov.
 The first asked for a structural result in the spirit of Corollary \ref{corAbel} (see also \cite{Glu25}).
 The second noticed that the constant $10$ in Theorem \ref{thMaxAS} is attained by $\Fi_{_{23}}$,
 and that the constant in Theorem \ref{thMain} is at least $5$ for almost simple groups.
	The third brought the article \cite{Bab22} to my attention, and helped me with the computer calculations in Lemma \ref{lemComp}.
	 After this, I did a substantial rewrite of the paper and had useful conversations with L. Babai.
	Finally, I am thankful to L. Pyber for the many ideas he has given me in these years.

 \vspace{0.1cm}
 \section{Breaking symmetry with set-stabilizers} \label{sec2} 
 
 Let $G$ be a finite group. The derived series is defined by $G^{(0)}=G$, and
 $$ G^{(n+1)} \> = \> [G^{(n)},G^{(n)}] $$
 for all $n \geq 0$.
 For a solvable group $G$, the {\itshape derived length} $dl(G)$ is the minimum $n$ such that $G^{(n)}=1$.
 For example, $dl(G)=1$ if and only if $G$ is abelian and $dl(G) \leq 2$ if and only if $G$ is metabelian.
 
  When $G$ acts on a finite set $\Omega$, then $G$ acts naturally on the power set $\mathcal{P}(\Omega)$.
   If $\Delta \subseteq \Omega$,
we write $C_G(\Delta)$ for the pointwise stabilizer, and $\Stab_G(\Delta)$ for the setwise stabilizer.
For each $g \in G$, it is easy to see that the stabilizers of $\Delta$ and $\Delta^g$ are conjugate subgroups.
 As usual, we say that $\Delta$ lies in a regular orbit (i.e., the orbit $\Delta^G$) if $\Stab_G(\Delta)=1$.
 Of course, as subgroups of $G$, the stabilizers still act on $\Omega$.
 
 \begin{definition}
  For $i \geq 1$,
  we say that $\Delta \subseteq \Omega$ lies in a {\itshape $i$-regular} orbit if all the orbits of $\Stab_G(\Delta)$ on $\Omega$ have length at most $i$.
  \end{definition}
  
  The $1$-regular orbits are precisely the genuine regular orbits.
  Moreover, unless explicitly stated otherwise, the $6$-regular orbits will also be called {\itshape good} orbits.
  (This is motivated by Theorem \ref{thSolvSet}.)
  We note that, by definition,
  $$ \mbox{ regular } \Rightarrow \mbox{ 2-regular } \Rightarrow \mbox{ 3-regular } \Rightarrow \mbox{ good } . $$
 
 \begin{remark} \label{remTotOrbits}
 The total number of orbits of $G$ on $\mathcal{P}(\Omega)$ is at least $|\Omega|+1$.
 \end{remark}
 
We now make a simple consideration.
 If $G \leqslant \Sym(\Omega)$ and $\Delta = \{x_1,\ldots,x_b\}$ is a base (i.e., $C_G(\Delta)=1$),
 then
  $$ \StabG(\Delta) \cong \StabG(\Delta)/C_G(\Delta) \leqslant \Sym(b) . $$
 So, if $\Stab_G(\Delta)$ is solvable, then $dl(\Stab_G(\Delta)) \leq  f(b)$, where
 $$ f(b) = \> \max \, \{ dl(S) : S \leqslant \Sym(b), S \mbox{ solvable} \} . $$ 
 We point out that $f(b) = O(\log b)$ \cite{Dix68}.
 When combined with \cite{Ser96}, this gives that in every primitive solvable group
 there exists a set-stabilizer of derived length at most $f(4)=3$.
 We are going to see that a little more is true.

 \subsection{Primitive solvable groups} 
 
Gluck \cite{Glu83} proved that, with finitely many exceptions, a primitive solvable group has a regular orbit on the power set.
 His result has been refined several times, for example in \cite{Ser97,MW04}.
 The following version is tailored for attacking Theorem \ref{thSolvSet}, and its proof is mostly computational.
  
 \begin{lemma} \label{lemPrimGood}
 Let $G \leqslant \Sym(\Omega)$ be a primitive solvable group.
 If $|\Omega| \geq 10$, then $G$ has at least $8$ regular orbits on $\mathcal{P}(\Omega)$.
 If $G$ has less than $5$ regular orbits on $\mathcal{P}(\Omega)$,
 then $G$ is one among the groups in Table \ref{table1}.
 \end{lemma}
 \begin{proof}
 The first part follows from \cite[Lem. 4.1(ii)]{MW04}.
 The remaining cases have been handled with the library \texttt{PrimitiveGroups} in \cite{GAP4}.
 \end{proof} 
 
   {\tiny
\begin{table}[ht] 
\centering
	
	\begin{tabular}{| c | c c c c |} 
	
		\hline
		
		\phantom{$\dfrac{1^1}{1_{1_1}}$}\hspace{-6mm}
		GAP Id & Regular orbits on $\mathcal{P}(\Omega)$ & $2$-regular orbits on $\mathcal{P}(\Omega)$ 
		& $3$-regular orbits on $\mathcal{P}(\Omega)$ & Good orbits on $\mathcal{P}(\Omega)$ \\
		
		\phantom{$\dfrac{1^1}{1_{1_1}}$}\hspace{-6mm}
		 & $\ell_{2,1}$ & $\ell_{2,2}$ & $\ell_{2,3}$ & $\ell_{2,6}$ \\
		
		\hline \hline
		
		\phantom{$\dfrac{1^1}{1_{1_1}}$}\hspace{-6mm}
		$[2,1]$ & $1$ & $3$ & $3$  & $3$\\
		
		\phantom{$\dfrac{1^1}{1_{1_1}}$}\hspace{-6mm}
		$[3,1]$ & $2$ & $2$ & $4$ & $4$ \\
		
		\phantom{$\dfrac{1^1}{1_{1_1}}$}\hspace{-6mm}
		$[3,2]$ & $0$ & $2$ & $4$ & $4$\\
		
		\phantom{$\dfrac{1^1}{1_{1_1}}$}\hspace{-6mm}
		$[4,1]$ & $0$ & $1$ & $3$ & $5$  \\
		
		\phantom{$\dfrac{1^1}{1_{1_1}}$}\hspace{-6mm}
		$[4,2]$ & $0$ & $1$ & $3$ & $5$ \\
		
		\phantom{$\dfrac{1^1}{1_{1_1}}$}\hspace{-6mm}
		$[5,2]$ & $0$ & $6$  & $6$ & $8$\\
		
		\phantom{$\dfrac{1^1}{1_{1_1}}$}\hspace{-6mm}
		$[5,3]$ & $0$ & $2$ & $2$ & $6$ \\
		
		\phantom{$\dfrac{1^1}{1_{1_1}}$}\hspace{-6mm}
		$[7,2]$ & $2$ & $16$ & $16$ & $16$ \\
		
		\phantom{$\dfrac{1^1}{1_{1_1}}$}\hspace{-6mm}
		$[7,3]$ & $4$ & $4$ & $10$ & $10$ \\
		
		\phantom{$\dfrac{1^1}{1_{1_1}}$}\hspace{-6mm}
		$[7,4]$ & $0$ & $4$ & $6$ & $8$  \\
		
		\phantom{$\dfrac{1^1}{1_{1_1}}$}\hspace{-6mm}
		$[8,1]$ & $3$ & $5$ & $5$ & $6$  \\
		
		\phantom{$\dfrac{1^1}{1_{1_1}}$}\hspace{-6mm}
		$[8,2]$ & $0$ & $0$ & $3$ & $6$ \\
		
		\phantom{$\dfrac{1^1}{1_{1_1}}$}\hspace{-6mm}
		$[9,2]$ & $0$ & $10$ & $10$ & $24$  \\
		
		\phantom{$\dfrac{1^1}{1_{1_1}}$}\hspace{-6mm}
		$[9,3]$ & $4$ & $6$ & $6$ & $14$ \\
		
		\phantom{$\dfrac{1^1}{1_{1_1}}$}\hspace{-6mm}
		$[9,4]$ & $4$ & $8$ & $8$ & $12$ \\
		
		\phantom{$\dfrac{1^1}{1_{1_1}}$}\hspace{-6mm}
		$[9,5]$ & $0$ & $4$ & $4$ & $12$ \\
		
		\phantom{$\dfrac{1^1}{1_{1_1}}$}\hspace{-6mm}
		$[9,6]$ & $0$ & $0$ & $4$ & $10$ \\
		
		\phantom{$\dfrac{1^1}{1_{1_1}}$}\hspace{-6mm}
		$[9,7]$ & $0$ & $0$ & $4$ & $10$ \\
		
		\hline
		
	\end{tabular}
	
	\medskip
	{\caption{Primitive solvable groups with less than $5$ regular orbits on $\mathcal{P}(\Omega)$.} \label{table1} }
\end{table}
} 
 
 As an example to help reading Table \ref{table1}, we remark that \texttt{PrimitiveGroup(4,2)} is $\Sym(4)$.
 There is no regular orbit of $\Sym(4)$ on $\mathcal{P}(\{1,2,3,4\})$,
 the stabilizer of $\{1,2\}$ has two orbits of length $2$,
 and the stabilizers of $\{1\}$ and $\{1,2,3\}$ have orbits of length $3$.
  We also note that \texttt{PrimitiveGroup(8,2)} is $\mathrm{A\Gamma L}_1(8)$, while \texttt{PrimitiveGroup(9,7)} is $\AGL_2(3)$.
 The following result is obtained with similar considerations:
 
 \begin{lemma}
  Let $G \leqslant \Sym(\Omega)$ be a primitive solvable group.
  The following hold:
  \begin{itemize}
  	\item There exists a subset of $\Omega$ whose setwise stabilizer has all orbits on $\Omega$ of length at most $3$;
  \item There exists a subset of $\Omega$ whose setwise stabilizer is metabelian.
  	If $G \neq \AGL_2(3)$, then there exists a subset whose setwise stabilizer is abelian.
  \end{itemize}
 \end{lemma}
 \begin{proof}
 As before, using \cite[Lem. 4.1(ii)]{MW04}, and the library \texttt{PrimitiveGroups} in \cite{GAP4}.
 We observe that $\mathrm{A\Gamma L}_1(8)$ and \texttt{PrimitiveGroup(9,6)} do not have any $2$-regular orbit on $\mathcal{P}(\Omega)$,
 but still have some abelian set-stabilizer.
 \end{proof}
 
 We conclude this subsection with the following result:
  
  \begin{lemma} \label{lemPrimDist} 
   If $G \leqslant \Sym(\Omega)$ is a primitive solvable group, then there exist disjoint subsets
   $\Delta_1,\Delta_2,\Delta_3$ such that
   $$ \cap_{j=1}^3 \, \StabG(\Delta_j) \> = \> 1 . $$
 \end{lemma}
 \begin{proof}
 This is \cite[Lem. 4.1(iii)]{MW04}.
 \end{proof}

 \subsection{Good colorings}
 
 The standard argument to deal with imprimitive groups is \cite[Th. 2]{Glu83},
 which was used by Gluck to show that permutation groups of odd order have a regular orbit on the power set.
 This does not work with a property such as ``have all orbits of length at most $C$'', where $C$ is a constant.
  We use a different method based on good orbits and partitions that we are going to explain.
  A key observation, already in the lines of the proof of \cite[Th. 12.1]{Bab22},
  is that studying stabilizers of partitions
  can also be useful to obtain results concerning stabilizers of a single subset.
  
 Fix $G \leqslant \Sym(\Omega)$, and let $[j] = \{1,\ldots,j\}$ (the set of colors).
 A {\itshape coloring} is a function $\Omega \to [j]$, and a $k$-coloring is a coloring with at most $k$ colors.
 (This can also be described by the preimages of the colors, which form a partition of $\Omega$.)
 A coloring is {\itshape asymmetric} if only the identity fixes every color.
 We refer the reader to \cite[Sec. 2]{Bab22} for some history concerning breaking-symmetry results and some conflicting terminology.
 
 \begin{definition} 
 For $i \geq 1$, an {\itshape $i$-asymmetric $k$-coloring} is a $k$-coloring
 such that the stabilizer of the coloring in $G$ has all orbits on $\Omega$ of length at most $i$.
 \end{definition} 
 
 The $1$-asymmetric colorings are precisely the genuine asymmetric colorings.
 Unless explicitly stated otherwise, the $6$-asymmetric colorings will also be called {\itshape good} colorings. 
 The total number of $k$-colorings of $\Omega$ is $k^{|\Omega|}$,
 but many of these are equivalent under the action of $G$.
 Fix $G \leqslant \Sym(\Omega)$; for each $k \geq 2$ we write
 $$ \ell_{k,1} , \hspace{0.5cm} \ell_{k,2} , \hspace{0.5cm} \ell_{k,3}, \hspace{0.5cm} \ell_{k,6}, $$
 for the numbers of $G$-orbits of asymmetric, $2$-asymmetric, $3$-asymmetric and good $k$-colorings.
 When necessary, we write $\ell_{k,i}(G)$ to stress the group $G$ we are referring to.
  It is obvious that, for every $k$ and $i \leq j$,
  \begin{equation} \label{eqTrivial}
  \ell_{k+1,i} \geq \ell_{k,i} , \hspace{1cm} \mbox{ and } \hspace{1cm} \ell_{k,j} \geq \ell_{k,i} .
  \end{equation}
 We observe that describing a $2$-coloring is equivalent to describing a subset.
 Therefore, $\ell_{2,1}$, $\ell_{2,2}$, $\ell_{2,3}$ and $\ell_{2,6}$ coincide respectively with
 the numbers of regular, $2$-regular, $3$-regular and good orbits on the power set.
 Restated in terms of colorings, Lemma \ref{lemPrimDist} says that $\ell_{4,1} \geq 1$ for every primitive solvable group
 (one color is needed for the complement).
 
A portion of the proof of Theorem \ref{thSolvSet} relies
  on showing that $\ell_{6-i,i} \geq 5$ for all primitive solvable groups and each $i=1,2,3$.
 The following combinatorial lemma is useful to estimate $\ell_{k,i}$ for large $k$.
 
 \begin{lemma} \label{lemNew} 
 Fix $k \geq 3$ and $i \in \{1,2,3,6\}$.
If $\ell_{2,i} \geq 1$, then $\ell_{k,i} \geq {k \choose 2}$.
If $\ell_{j,i} > \ell_{j-1,i}$ for some $3 \leq j \leq k$, then $\ell_{k,i} \geq {k \choose j}$.
 \end{lemma}
 \begin{proof}
 The strict inequality $\ell_{j,i} > \ell_{j-1,i}$ means that there exists a suitable coloring that involves precisely $j$ colors.
Now suppose we have such a $j$-coloring.
For each subset of cardinality $j$ of the $k$ available colors, we obtain an inequivalent suitable $j$-coloring.
 \end{proof} 
 
 \begin{corollary} \label{corNew}
 Let $A$ be a primitive solvable group. Then $\ell_{5,1}(A) \geq 5$.
 \end{corollary}
 \begin{proof}
 Fix $A$. We have $\ell_{4,1}(A) \geq 1$ by Lemma \ref{lemPrimDist}.
 This means that $\ell_{2,1}(A) \geq 1$, or the inequality $\ell_{j,1}(A) > \ell_{j-1,1}(A)$ occurs at some index $j=3,4$.
 Considering the worst case, we obtain
 $$ \ell_{5,1}(A) \geq \min_{j = 2,3,4} {5 \choose j}  = 5 , $$ 
 where we used Lemma \ref{lemNew} with $k=5$, $i=1$.
 \end{proof}
 
 The next result concerns $3$-colorings.
 
 \begin{lemma} \label{lemComp} 
 Let $A$ be a primitive solvable group.
 If $A \neq 1$, then $\ell_{3,2}(A) \geq 5$.
 \end{lemma}
 \begin{proof}
 Let $\Omega$ be the base set, so that $|\Omega| \geq 2$.
 By Lemma \ref{lemPrimGood} and (\ref{eqTrivial}), we can assume $|\Omega| \leq 9$.
 Since we are interested in a lower bound on the number of orbits, it is enough to deal with the solvable (over)groups
 $$ \Sym(2) , \hspace{0.2cm} \Sym(3) , \hspace{0.2cm} \Sym(4) , \hspace{0.2cm} \mathrm{AGL}_1(5) , \hspace{0.2cm} \mathrm{AGL}_1(7) ,
  \hspace{0.2cm} \mathrm{A\Gamma L}_1(8) , \hspace{0.2cm}  \mathrm{AGL}_2(3)  . $$ 
 It is easy to see that $\ell_{3,2}$ is at least $5$ (actually $6$) in all these cases.
 In particular, using \cite{GAP4}, the command
 \begin{center}
 \texttt{OrbitsDomain( A , Tuples([1..k], NrMovedPoints(A)) , Permuted )} 
 \end{center}
 returns the orbits of $A$ on the $k$-colorings of $\Omega$.
 \end{proof}

  \subsection{Imprimitive groups} \label{sec2.3}
  
  We now explain the general strategy to deal with colorings of imprimitive groups.
 Let $A$ and $B$ be permutation groups on the finite sets $\Omega_1$ and $\Omega_2$ respectively.
 Let $G \leqslant B \wr_{|\Omega_1|} A$ act imprimitively on $\Omega = \Omega_1 \times \Omega_2$.
 Let $x,y \geq 2$ be integers.
  Let $X = (X_i)_{i=1}^x$ be a $x$-coloring of $\Omega_1$, namely $\Omega_1 = \sqcup_{i = 1}^x X_i$.
 Moreover, suppose we have $x$ inequivalent $y$-colorings of $\Omega_2$, say $Y_1,\ldots,Y_x$.
 In particular, for each $i = 1,\ldots,x$, we have that $\Omega_2 = \sqcup_{j=1}^y Y_{i,j}$.
 The point is to use the $y$-colorings of $\Omega_2$ as colors for $\Omega_1$.
 (In fact, we intentionally used the same index $i$ to count the colors in $X$ and the colorings of $\Omega_2$).
 We will construct a new $y$-coloring of $\Omega$, say $Z = (Z_j)_{j=1}^y$, as follows.
 For each $j= 1,\ldots y$, define
 $$ Z_j \> = \>
 \{ (\omega_1,\omega_2) \in \Omega : 
 \omega_1 \in X_i \mbox{ and } \omega_2 \in Y_{i,j} \mbox{ for some fixed } i \} . $$
 It is easy to see that $(Z_j)_{j=1}^y$ is a partition of $\Omega$.
 
 \begin{lemma} \label{lemStabEmb}
 The stabilizer of the coloring $Z$ in $G$ can be embedded in the semidirect product
 between the direct product of the stabilizers of the colorings $Y_i$ in $B$
 and the stabilizer of the coloring $X$ in $A$.
 More precisely,
 $$ \cap_j \StabG(Z_j) \> \leqslant \> 
 \left( \prod_i \left( \cap_j \StabB(Y_{i,j}) \right) \right) \rtimes \left(\cap_i \StabA(X_i) \right) . $$
 \end{lemma} 
 \begin{proof}
 First suppose $G=B \wr_{|\Omega_1|} A$.
 Since the colorings of $\Omega_2$ are inequivalent, the stabilizer of $Z$ is precisely the above semidirect product.
 The general case follows.
 \end{proof} 
 
 \begin{lemma} \label{lemIneq}
  Fix $Y_1,\ldots,Y_x$ colorings of $\Omega_2$, and let $X,X'$ be colorings of $\Omega_1$.
 Let $Z$ and $Z'$ be the colorings of $\Omega_1 \times \Omega_2$ obtained from $X$ and $X'$ respectively.
 If $X$ and $X'$ are inequivalent, then $Z$ and $Z'$ are inequivalent.
 \end{lemma}
 \begin{proof} 
 If $Z$ and $Z'$ lies in the same orbit of $B \wr_{|\Omega_1|} A$,
 then the induced colorings of $Z$ and $Z'$ on $\Omega_1$ need to be in the same orbit of $A$.
 But these restrictions are precisely $X$ and $X'$, and the proof follows.
 \end{proof}
 
 We are ready to prove Theorem \ref{thSolvSet}.
 We actually prove the following stronger statement, that lends itself to an inductive argument:
 
  \begin{theorem} \label{thStrong}
 Let $G \leqslant \Sym(\Omega)$ be a solvable permutation group.
If $|\Omega| \geq 2$, then there are at least $3$ good orbits on $\mathcal{P}(\Omega)$.
  If there are less than $5$ good orbits on $\mathcal{P}(\Omega)$, then $|\Omega| \leq 3$.
 \end{theorem} 
\begin{proof} 
We work by induction on the cardinality of $\Omega$.
Suppose that $\Omega_1,\ldots,\Omega_t$ are the orbits of $G$ on $\Omega,$
and that we can find subsets $\Delta_j \subseteq \Omega_j$ such that $\Stab_G(\Delta_j)$
has orbits of length at most $6$ on $\Omega_j$, for each $j=1,\ldots,t$.
Then $\Stab_G( \sqcup_j \> \Delta_j )$ will surely work.
It follows that we can suppose that $G$ is transitive.
The base cases for the induction, primitive groups, can be handled with Lemma \ref{lemPrimGood} and Table \ref{table1}.
 Let $\Omega = \Omega_1 \times \Omega_2$, with $|\Omega_1| \geq 2$.
 Let $G \leqslant B \wr_{|\Omega_1|} A$, with $A$ primitive on $\Omega_1$, and $B$ transitive on $\Omega_2$.
 By the inductive hypothesis applied to $\Omega_2$, we are in one of the following cases:
 
 \begin{itemize}
 \item \underline{$B$ has at least $5$ good orbits on $\mathcal{P}(\Omega_2)$}.
 
 Since the goal is good $2$-colorings, we will use the construction above with $y=2$.
 In this case, set $x=5$.
 Let $\Delta_1, \ldots , \Delta_5 \subseteq \Omega_2$ be such that
 $\Delta_1^B,\ldots,\Delta_5^B$ are distinct good orbits of $B$ on $\mathcal{P}(\Omega_2)$.
 For each asymmetric $5$-coloring of $\Omega_1$, say $X = (X_i)_{i = 1}^5$,
 we construct a good orbit $\Sigma^G$ for $G$ as follows:
 $$ \Sigma \> = \> \{ (\omega_1,\omega_2) \in \Omega_1 \times \Omega_2 : 
 \omega_1 \in X_i \mbox{ and } \omega_2 \in \Delta_i \mbox{ for some fixed } i  \} . $$ 
 In fact, by Lemma \ref{lemStabEmb}, the set-stabilizer of $\Sigma$ in $G$ can be embedded in the semidirect product
 between the direct product of the set-stabilizers of the $\Delta_i$'s in $B$ (which have orbits of length at most $6$),
 and the stabilizer of the asymmetric $5$-coloring in $A$ (which is trivial).
 It follows from Lemma \ref{lemIneq} that there are at least $\ell_{5,1}(A)$ good orbits of $G$.
 The proof of this part is completed by Corollary \ref{corNew}.
 
   \item \underline{$B \cong \Sym(2)$}.
   
   Here $\ell_{2,2}(B)=3$.
   Using the $2$-asymmetric orbits of $B$, and the $3$-asymmetric $3$-colorings of $\Omega_1$, arguing as before
we can construct $\ell_{3,3}(A)$ good orbits of $G$ (set $x=3$, $y=2$ in the construction above).
In fact, note that orbit lengths are multiplied in the imprimitive action, so the length of each orbit will be bounded by $2 \cdot 3 =6$.
  Now $\ell_{3,3}(A) \geq \ell_{3,2}(A) \geq 5$ by (\ref{eqTrivial}) and Lemma \ref{lemComp}. 
 
  \item \underline{$B \leqslant \Sym(3)$}.
  
  We can assume $B \cong \Sym(3)$, so that $\ell_{2,3}(B) = 4$.
  Using the $3$-asymmetric orbits of $B$, and the $2$-asymmetric $4$-colorings of $\Omega_1$, as before
we can construct $\ell_{4,2}(A)$ good orbits of $G$ (set $x=4$, $y=2$ in the construction above).
The length of each orbit will be bounded by $3 \cdot 2 =6$.
  Finally, $\ell_{4,2}(A) \geq \ell_{3,2}(A) \geq 5$ by (\ref{eqTrivial}) and Lemma \ref{lemComp}.
 \end{itemize} 

The proof is complete.
\end{proof} 

The cardinality of the subset in Theorem \ref{thSolvSet} grows linearly with the cardinality of $\Omega$,
 which is necessary as shown by $B \wr K$,
 when $B$ has some orbit of length exceeding $6$, and the degree of $K$ grows.

\begin{proof}[Proof of Corollary \ref{corMain}]
Let $\Delta \subseteq \Omega$ be as provided by Theorem \ref{thSolvSet}.
Then $\Stab_G(\Delta)$ can be embedded in a direct product of copies of $\Sym(6)$.
Every solvable subgroup of $\Sym(6)$ has derived length at most $3$, and the proof follows.
\end{proof}

The next example provides a transitive solvable group where no subset gives a metabelian stabilizer.

 \begin{example} \label{exBad}
   Let $B = \Sym(4) \wr_4 \Sym(4)$ act imprimitively on $\Omega = \{a,b,c,d\} \times \{1,2,3,4\}$.
   It is not hard to see that $B$ has a unique orbit on the power set whose set-stabilizer is metabelian,
   namely $\Delta^B$, where
   $$ \Delta \> = \> \{ (a,1) , (b,1) ,  (b,2) , (b,3) ,  (c,1) , (c,2) ,  (d,1) , (d,2) \} . $$
   The set-stabilizer $H$ of $\Delta$ in $B$ is isomorphic to $\Sym(3) \times \Sym(3) \times (C_2 \times C_2) \wr_2 C_2$.
   Now let $G = B \wr_2 \Sym(2)$ act on $\Omega \times \{ \alpha , \beta \}$.
   If we color both $\Omega \times \{ \alpha \}$ and $\Omega \times \{ \beta \}$ using $\Delta$,
   then the resulting set-stabilizer will be isomorphic to $H \wr_2 \Sym(2)$, whose derived length is $3$.
   Otherwise, we are forced to color a subset, say of $\Omega \times \{ \alpha \}$,
   which does not provide a metabelian set-stabilizer in $B$.
   So there is no hope of getting a metabelian set-stabilizer in $G$.
   \end{example}

Since $\Sym(3) \wr_2 \Sym(2)$ is not supersolvable, neither is $H \wr_2 \Sym(2)$,
   and the same example shows that supersolvable set-stabilizers need not to exist.
   This is an improvement on \cite[Ex. 3.4]{Glu25}, which deals with nilpotent set-stabilizers.
The next similar construction involves $\mathrm{A\Gamma L}_1(8)$,
and provides a transitive solvable group where no subset gives a stabilizer with all orbits of length at most $5$.
We observe that, if a finite group acts on a set (even if not transitively), then the length of each orbit has to divide the order of the group.
So, for example, $\mathrm{A\Gamma L}_1(8)$ does not have any orbit of length $5$.

\begin{example} \label{exBad2}
	We first notice that $\mathrm{A\Gamma L}_1(8)$ has precisely $4$ orbits on the power set whose stabilizers have orbits of length at most $4$,
	and $4$ is attained by one of these orbits (a crucial fact is that $\ell_{2,3}(\mathrm{A\Gamma L}_1(8)) =3$ in Table \ref{table1}).
   Let $B = \mathrm{A\Gamma L}_1(8) \wr_4 \Sym(4)$ act imprimitively on a set $\Omega$ of cardinality $32$.
   It is not hard to see that $B$ has a unique orbit $\Delta^B$ on the power set where the set-stabilizer has orbits of length at most $5$ (actually $4$).
   Now let $G = B \wr_2 \Sym(2)$ act on $\Omega \times \{ \alpha , \beta \}$.
   If we color both $\Omega \times \{ \alpha \}$ and $\Omega \times \{ \beta \}$ using $\Delta$,
   then the resulting set-stabilizer will have an orbit of length $4 \cdot 2 =8$.
   Otherwise, we are forced to color a subset, say of $\Omega \times \{ \alpha \}$,
   which does not provide a set-stabilizer (in $B$) with all orbits of length at most $5$.
   So there is no hope of getting such a stabilizer in $G$.
   \end{example}

\subsection{A result about $3$-colorings}

With the same ideas developed in the proof of Theorem \ref{thStrong},
we obtain the following powerful result concerning $3$-colorings.

\begin{theorem} \label{th3Col}
Let $G$ be a solvable permutation group on the finite set $\Omega$, $|\Omega| \geq 2$.
Then there exist at least $5$ orbits of $3$-colorings of $\Omega$ whose stabilizers have all orbits of length at most $2$.
\end{theorem}
\begin{proof}
We have to prove that $\ell_{3,2}(G) \geq 5$.
We work by induction on $|\Omega|$, and we can suppose that $G$ is transitive.
Also, primitive groups can be handled with Lemma \ref{lemComp}.
 Let $\Omega = \Omega_1 \times \Omega_2$, with $|\Omega_1| \geq 2$.
 Let $G \leqslant B \wr_{|\Omega_1|} A$, with $A$ primitive on $\Omega_1$, and $B$ transitive on $\Omega_2$.
 By the inductive hypothesis, we have that $\ell_{3,2}(B) \geq 5$.
 We argue as in Subsection \ref{sec2.3}, with $x=5$ and $y=3$.
 Using the $2$-asymmetric $3$-colorings of $\Omega_2$ as colors for $\Omega_1$,
 and the asymmetric $5$-colorings of $\Omega_1$,
 we can construct $\ell_{5,1}(A)$ inequivalent $2$-asymmetric $3$-colorings of $\Omega$.
 We have $\ell_{5,1}(A) \geq 5$ from Corollary \ref{corNew},
 which makes the induction smooth and completes the proof.
\end{proof} 

The correspondent structural result answers \cite[Quest. 7c]{Bab22}.

\begin{corollary} \label{corAbel}
Let $G$ be a solvable permutation group on the finite set $\Omega$.
Then there exists a $3$-coloring of $\Omega$ whose stabilizer is an elementary abelian $2$-group.
\end{corollary}

See \cite{Glu25} for another positive result (with two colors) that emphasizes the importance of the prime $2$.
In general, five colors are necessary and sufficient to achieve a trivial stabilizer \cite[Th. 1.2]{Ser96}.

\subsection{CFSG-free results for nonsolvable groups}

The foundation stone to prove a result like Theorem \ref{thSolvSet} is the fact that
a primitive solvable group of large degree has a regular orbit on the power set \cite{Glu83}.
With the trivial exceptions of the alternating and symmetric groups,
the same is true in the nonsolvable case, but the known proofs of Cameron-Neumann-Saxl and Seress \cite{CNS84,Ser97}
require the Classification of Finite Simple Groups.
Actually, Cameron \cite{Cam86} proved a stronger statement concerning the number of regular orbits.
We observe that a more recent elementary result of Sun and Wilmes \cite{SW15,SW16}
is strong enough to provide an easier proof of Cameron's theorem.\footnote{Babai pointed out that
a slightly weaker observation, which gives an elementary proof of \cite{CNS84}, was made by himself in \cite[Th. 15.1]{Bab22}.}

\begin{theorem} \label{thCam}
Let $G$ be a primitive group of degree $n$ not containing $\Alt(n)$.
Then the proportion of regular orbits on the power set tends to $1$ as $n \to \infty$.
\end{theorem}
\begin{proof}
By the remark near the bottom of page 308 of \cite{Cam86},
it suffices to show that $|G| \leq c^{\sqrt{n}}$ for some absolute constant $c$
- this is immediate from \cite[Cor. 1.6]{SW16}.
\end{proof}

We stress that Babai's theorem \cite{Bab81} has several applications,
and Theorem \ref{thCam} is a nice example where Sun-Wilmes improvement turns out to be crucial.
(See \cite{Ebe24} for another recent example.)

We are ready to prove Theorem \ref{thNonAlt}. Again, we prove a slightly stronger statement.

\begin{theorem} \label{thCam+}
There exists a constant $C$ such that the following holds.
Let $G$ be a permutation group on the finite set $\Omega$, $|\Omega| \geq 2$.
Then there are at least two orbits of subsets of $\Omega$
whose setwise stabilizers have all nonalternating composition factors of order at most $C$.
\end{theorem}
 \begin{proof} 
 After Theorem \ref{thCam}, it is easy to see that such a constant $C$ exists for all primitive groups.
 Given $C$, we call {\itshape nice} the orbits on the power set with the property in the statement.
 We can suppose that $G$ is transitive.
 By the Jordan-H\"older theorem, the property of having all nonalternating composition factors of order at most $C$,
 is a group property preserved by group extensions.
Let $\Omega = \Omega_1 \times \Omega_2$, with $|\Omega_1| \geq 2$.
 Let $G \leqslant B \wr_{|\Omega_1|} A$, with $A$ primitive on $\Omega_1$, and $B$ transitive on $\Omega_2$.
 By induction, $B$ admits two nice orbits on $\mathcal{P}(\Omega_2)$ (equivalently, two nice $2$-colorings).
 Arguing as in Subsection \ref{sec2.3} (set $x=y=2$), using these nice orbits of $B$ as colors for $\Omega_1$,
 and the two nice $2$-colorings of $\Omega_1$,
 we can construct two inequivalent nice orbits of $G$ on $\mathcal{P}(\Omega)$.
 The proof is complete.
\end{proof}

Using the Classification of Finite Simple Groups,
it should be possible to find an explicit version of Theorem \ref{thCam+}
(it is enough to find $C$ for all primitive groups).

 \vspace{0.1cm}
 \section{Primitive groups with solvable stabilizer} \label{sec3} 
 
 In this section, we move to consider pointwise stabilizers of very small sets.
 We would like to remark that, the idea of fixing a bounded number of points to obtain something interesting,
 is completely hopeless in imprimitive groups.
 This is a key difference with respect to stabilizing sets (see Corollaries \ref{corMain} or \ref{corAbel}, for example).
 
 \begin{example} 
 Let $B$ be a transitive group with stabilizer $B_0$, and let $S \leqslant \Sym(d)$.
 Let $G = B \wr_d S$ act imprimitively on $|B:B_0|d$ points.
 The stabilizer of a point in $G$, say $H$, contains a copy of $B_0 \times B^{d-1}$.
For each $b \geq 1$, the intersection of any $b$ conjugates of $H$ will contain a copy of $B^{d-b}$,
 showing that fixing less than $d$ points is useless.
 \end{example} 
 
 Therefore, we will focus on primitive groups, with the goal of proving Theorem \ref{thMain}.
 This proof is more algebraic, but the combinatorial ideas in the previous section will play a key role.
 In the case of primitive solvable groups, Theorem \ref{thMain} is \cite[Th. 1.1]{LS25}.
 This part is really about finding a good centralizer in irreducible solvable linear groups,
 which was done by Yang \cite{Yan09} using ideas in \cite{MW04}.
Using the O'Nan-Scott theorem, as explained by Li and Zhang in \cite[Sec. 2]{LZ11},
it is not hard to see that a primitive nonsolvable group with solvable stabilizer is almost simple or of a specific product type.
 
\subsection{Almost simple groups} 

Theorem \ref{thMaxAS} will follow from the tables in \cite[Sec. 10]{LZ11},
together with a few properties of finite simple groups and linear groups.
We first report a strong version of the Schreier conjecture,
which says that the group $\Out(T)$ of the outer automorphisms of a finite simple group is solvable.

\begin{lemma} \label{lemOut}
If $T$ is a finite simple group, then $\Out(T)$ is solvable of derived length at most $3$.
More precisely:
\begin{itemize}
\item[(i)] If $T$ is cyclic, alternating, or a sporadic simple group, then $\Out(T)$ is abelian;
\item[(ii)] $dl(\Out(T)) =3$ if and only if $T$ is the orthogonal group $\mathrm{P}\Omega^+_8(q)$ for $q$ odd.
\end{itemize}
\end{lemma}
\begin{proof}
This follows from the Classification of Finite Simple Groups.
In particular, see \cite[Pag. xvi, Table 5]{CCN+85}.
For finite simple groups of Lie Type, $\Out(T)$ can be described as an extension of three groups:
diagonal, field and graph automorphisms (the field and graph automorphisms commute).
These groups are abelian except for the graph automorphisms group of $\mathrm{P}\Omega^+_8(q)$,
which is isomorphic to $\Sym(3)$.
\end{proof}

\begin{remark} \label{remSolRad}
In an almost simple group $G$, the socle $T$ is the unique minimal normal subgroup.
This implies that the solvable radical of $G$ is trivial and that every solvable subgroup of $G$ is core-free.
\end{remark} 

\begin{proof}[Proof of Theorem \ref{thMaxAS}]
Let $T$ be the socle of $G$, and let $G_0$ be the normal subgroup containing $T$ which is described in \cite[Pag. 444]{LZ11}.
In particular, $G_0$ is minimal such that $H_0 = H \cap G_0$ is a maximal subgroup of $G_0$.
It follows that the maximal subgroup $H$ of $G$ is an extension of $H_0$ with
$$ H/H_0 \cong G_0H/G_0 = G/G_0 , $$
which is a section of $G/T$ and so of $\Out(T)$.
Therefore
$$ dl(H) \leq dl(H_0) + dl(G/G_0) . $$
Looking at the Tables 14-20 in \cite{LZ11}, we observe that one of the following holds:
\begin{itemize}
\item[(i)] $G_0$, and so $H_0$, is one of finitely many groups;
\item[(ii)] $G_0$ is symmetric or alternating, and $H_0$ has derived length bounded by $6$;
\item[(iii)] $G_0$ is a simple group of Lie type of bounded rank (recall that exceptional groups lie in this class);
\item[(iv)] $G_0$ is $\PSL_r(q)$ or $\PSU_r(q)$, and $H_0$ is metabelian.
\end{itemize}
To deal with (iii) uniformly, we use the fact that a solvable linear group
has derived length bounded by a function on the rank \cite{Dix68}.
This and Lemma \ref{lemOut} prove Theorem \ref{thMaxAS} for some absolute constant.
The (relatively high) constant $10$ is attained by the Fischer group $\Fi_{_{23}}$,
which makes smooth the remaining computations in (i)-(iii).
A more careful inspection of the tables shows that $dl(H_0) \leq 6$ in the infinite families of (iii),
while we use a computer for the remaining groups.
We point out that there are some mistakes
in \cite[Tab. 15]{LZ11}, and rely instead on the \cite{GAP4} library \texttt{CTblLib (v1.3.9)}.\footnote{One
mistake is in correspondence of $\Fi_{_{23}}$.
In fact, $\Fi_{_{23}}$ has a unique conjugacy class of solvable maximal subgroups, namely \texttt{AtlasSubgroup(``Fi23'',7)}),
with structure $3^{1+8}.2^{1+6}.3^{1+2}.2S_4$ (with the ATLAS notation) and derived length $10$.}
A potential obstruction is given by a certain conjugacy class of solvable maximal subgroups $H$ of the baby monster,
which seems to be not available on a computer.
With the ATLAS notation \cite{CCN+85}, $H$ is isomorphic to $[3^{11}].(S_4 \times 2S_4)$.
Any group of order $3^{11}$ has nilpotency class at most $10$ and derived length at most $4$,
showing that $dl(H) \leq 8$ in this case.
\end{proof} 
 
 A cheap version of Theorem \ref{thMaxAS} says that
there exists a constant $C$ such that if $G$ is a finite group with a solvable maximal subgroup $H$,
and the derived length of $H$ is more than $C$, then $G$ is not simple.
It would be interesting to find an elementary proof of this curious fact.

 \subsection{Groups of product type} 
 
Let $X$ be a primitive almost simple group with solvable point stabilizer $Y <X$, and let $T \lhd X$ be its socle.
Let $d \geq 2$, and let $S \leqslant \Sym(d)$ be a transitive solvable group.
By \cite[Th. 1.1]{LZ11}, the primitive group $G$ that we care about satisfies
$$ T \wr_d S \> \leqslant \> G \> \leqslant \> X \wr_d S , $$
and acts on $|X:Y|^d$ points.
(See also \cite[Sec. 8.2]{Bur21}.)

\begin{proof}[Proof of Theorem \ref{thMain}] 
By the discussion above, it rests only to deal with the product type case.
Let $V \cong X^d$ be the base subgroup of $X \wr_d S$, and let $H < G$ be the stabilizer of a point in the product action.
Let $W \cong Y^d$ be the subgroup of $V$ such that $H=WS \cap G$.
Stabilizing two points corresponds to intersecting $H$ with one of its conjugates.
To prove the theorem, we will find $v \in G \cap V$ such that the derived length of
$$ I_v \> = \> W S \cap (W S)^v $$
is bounded by an absolute constant (actually $13$).
The proof will follow, because $H \cap H^v \leqslant I_v$,
and the derived length of a subgroup is bounded by the derived length of the group.
Let $\pi \colon G \to S$ be the natural projection, and let $v \in V$.
Since $\Ker(\pi) = G \cap V$, we have
\begin{equation} \label{eqIv}
dl(I_v) \> \leq \> dl(I_v \cap V) + dl( \pi(I_v) ) . 
\end{equation}
We observe that
$$ I_v \cap V \subseteq WS \cap V  = W . $$
This gives $I_v \cap V=I_v \cap W$, and therefore we can write
$$ dl(I_v \cap V) \leq dl(W) = dl(Y) . $$
Now $dl(Y) \leq 10$ by Theorem \ref{thMaxAS}, and the rest of the proof deals with $\pi(I_v)$.
It is not hard to see that $C_S(v) = S \cap S^v \subseteq \pi(I_v)$, but in general $\pi(I_v)$ can be much bigger.
For every $w \in W$, $\sigma \in S$ and $v \in V$, an explicit computation gives
$$ (w,\sigma)^v = (v^{-1} w v^\sigma , \sigma) . $$
This implies that
$$ \pi(I_v) \> = \> \{ \sigma \in S : \> W \cap v^{-1} W v^\sigma \neq \emptyset \} . $$
Now let $v = (x_1,\ldots,x_d)$. We have 
$$ \pi(I_{(x_1,\ldots,x_d)})  =  \{ \sigma \in S : \> Y \cap x_i^{-1} Y x_{\sigma(i)} \neq \emptyset \mbox{ for each } i = 1, \ldots , d \} . $$
Since
$$ Y \cap x_i^{-1} Y x_{\sigma(i)} \neq \emptyset \hspace{1cm} \mbox{ if and only if } \hspace{1cm} x_{\sigma(i)} \in  Y x_i Y , $$
we obtain the key equality
\begin{equation} \label{eqKey}
 \pi(I_{(x_1,\ldots,x_d)}) \> = \> \{ \sigma \in S : \> x_{\sigma(i)} \in  Y x_i Y \mbox{ for each } i = 1, \ldots , d \} .
\end{equation}
	This suggests to look at the double cosets of $Y$ in $X$ as colors,
	and connects Theorem \ref{thMain} with the matter of Section \ref{sec2}.
	Let $c_1, \ldots ,  c_r \in X$ be representatives of these double cosets ($r \geq 2$).
	For any $v = (x_1,\ldots,x_d) \in V$, and $j \in \{1, \ldots, r \}$, let
	$$ \Delta_j = \Delta_j (v) = \{ i \in \{1,\ldots,d\} : x_i \in Yc_jY \} . $$
	Of course the $\Delta_j$'s form a partition of $\{1,\ldots,d\}$.
	The equality in (\ref{eqKey}) shows that $\pi(I_v)$ only depends on this partition,
	and in particular
	\begin{equation} \label{eqFin}
	\pi(I_v) \> = \> \cap_{j =1}^r \StabS(\Delta_j) \> = \> \cap_{j =1}^{r-1} \StabS(\Delta_j) . 
	\end{equation} 
	To construct $v$, we will use only two elements $c_1$ and $c_2$
	(and so two sets $\Delta_1$ and $\Delta_2 = \{1,\ldots,d\} \setminus \Delta_1$).
	One reason is that we want $v$ to lie in $G \cap V$, and not just in $V$.
	This is assured if we choose representatives $c_i$'s which are in $T$.
	We can surely choose $c_1=1$ and $c_2 \in T \setminus Y$, because otherwise $T \leqslant Y$,
	which is impossible because $Y$ is core-free.
	Finally, if $\Delta_1 = \Delta \subseteq \{1,\ldots,d\}$ is as provided by Corollary \ref{corMain},
	from (\ref{eqFin}) we get $dl( \pi(I_v) ) \leq 3$, concluding by (\ref{eqIv}) that $dl(I_v) \leq 13$ as desired.
\end{proof}
 
 \begin{remark}
	In the proof of Theorem \ref{thMain}, it might seem wasteful to construct $v$ with only two elements $c_1$ and $c_2$.
	As explained, one reason is that we want $v \in G$, but there is also a different motivation.
	In fact, it is well known that the number $r$ of double cosets of $Y$ in $X$
	equals the rank of the primitive group $X$ with solvable stabilizer $Y$,
	and there actually exist various doubly transitive groups of this type.
	A more detailed analysis of almost simple groups, and the additional use of Corollary \ref{corAbel},
	might help to find the right constant in Theorem \ref{thMain}.
	In the case of primitive solvable groups, this lies somewhere between $4$ and $9$
	\cite[Rem. 2.9]{LS25}.
	H.Y. Huang pointed out that $5$ is attained by the almost simple group $\Aut(\mathrm{P}\Omega^+_8(3))$.
	With some more work, it is not hard to use this information to construct a group of product type where
	the constant $6$ is achieved.
 \end{remark}
 
 A small corollary on set-stabilizers can be deduced from Theorem \ref{thMain}.
 
\begin{corollary}
Let $\delta \in \{2,3,4\}$.
Let $G \leqslant \Sym(\Omega)$ be a primitive group with solvable stabilizer.
Then there exists $\Delta \subseteq \Omega$, $|\Delta|=\delta$,
such that $\Stab_G(\Delta)$ is solvable of derived length bounded by an absolute constant.
\end{corollary}
\begin{proof}
Let $\{x_1,x_2\} \subseteq \Omega$ be as provided by Theorem \ref{thMain}.
If $\{x_1,x_2\} \subseteq \Delta$, then $C_G(\Delta)$ is solvable of bounded derived length.
Now $\Stab_G(\Delta)/C_G(\Delta) \leqslant \Sym(4)$,
and it follows that $\Stab_G(\Delta)$ is solvable of bounded derived length.
\end{proof}

   \vspace{0.5cm}

\end{document}